\documentclass[12pt,reqno]{amsart}
\usepackage{fullpage}
\usepackage[top=1.5in, bottom=2in, left=2.5in, right=2.5in]{geometry}
\usepackage{graphicx,amsmath,amssymb,amsfonts,amsthm,enumitem,bm,xcolor}
\usepackage{fullpage}
\usepackage{cleveref}

\usepackage{amsmath}

\usepackage[normalem]{ulem}

\usepackage{url}

\usepackage{color}

\renewcommand{\a}{\alpha}

\renewcommand{\r}{{\rho}}

\renewcommand{\(}{\left\(}
\renewcommand{\)}{\right\)}

\numberwithin{equation}{section}
\theoremstyle{plain}
\newtheorem{theorem}{Theorem}[section]

\newtheorem{conjecture}[theorem]{Conjecture}
\newtheorem*{conjecture*}{Conjecture}

\newtheorem{proposition}[theorem]{Proposition}

\theoremstyle{definition}

\newtheorem*{remark}{Remark}

\numberwithin{equation}{section}

\renewcommand{\pmod}[1]{\ \left( \mathrm{mod} \, #1 \right)}

\setlist[enumerate]{leftmargin=*,label=\rm{(\arabic*)}}

\makeatletter
\@namedef{subjclassname@2020}{%
	\textup{2020} Mathematics Subject Classification}
\makeatother

\allowdisplaybreaks

\title{Proof of a conjecture of Andrews and El Bachraoui on the parity of two-color partitions}
\author{Koustav Banerjee}
\author{Kathrin Bringmann}
\address{University of Cologne, Department of Mathematics and Computer Science, Weyertal 86-90, 50931 Cologne, Germany}
\email{kbanerj1@uni-koeln.de}
\email{kbringma@math.uni-koeln.de}

\makeatletter
\@namedef{subjclassname@2020}{%
	\textup{2020} Mathematics Subject Classification}
\makeatother

\subjclass[2020]{05A17,\ 11P81,\ 11P83.}
\keywords{colored partitions, congruences, $q$-series, representation numbers, theta functions}

\begin{document}

\begin{abstract}

 In this paper, we prove a conjecture of Andrews and El Bachraoui concerning the parity of certain two-color partitions. Precisely, we show that if the Fourier coefficient $t_o(n)$ of the corresponding $q$-series is odd, then $8n+9$ is represented by the binary quadratic form $x^2+2y^2$. %This settles a conjecture of Andrews and Bachraoui. The proof is elementary and relies only on classical $q$-series identities.
\end{abstract}

\maketitle

\section{Introduction and statement of results}
  Andrews and El Bachraoui \cite{AB} introduced certain restricted two-color partitions. Precisely, they considered two-color partitions of $n$, denoted by $s_1(n)$, into distinct parts whose smallest part occurs in one prescribed color only, while every larger part may occur in either color or in both colors. The corresponding generating function is given by\footnote{In \cite[p. 68]{AndCC}, Andrews interpreted $S_1(q)$ as the generating function of so-called strictly concave compositions, denoted by $V_d(n)$.} 
\[
S_1(q):=\sum_{n\ge 0}s_1(n)q^n=\sum_{n\ge 0}\left(-q^{n+1}\right)^2_{\infty}q^n.
\]
Here, for $n\in \mathbb{N}_0\cup \{\infty\}$, $(a)_n=(a;q)_n:=\prod_{j=0}^{n-1}(1-aq^j)$ denotes the usual {\it $q$-Pochhammer symbol}. Andrews and El Bachraoui \cite[Theorem 1]{AB} proved the following congruence of $s_1(n)$ modulo $4$:
\begin{theorem}
For $n\in \mathbb{N}_0$, we have
\[
s_1(n)\equiv \begin{cases}
(-1)^r\pmod{4} &\text{if}\quad n=\frac{r(r+1)}{2}\ \text{for some}\ r\in \mathbb{N}_0,\\
0\pmod{4} &\text{otherwise}.
\end{cases}
\]
In particular, $s_1(n)$ is odd iff $n$ is a triangular number.
\end{theorem} 
Moreover, they \cite[p. 2]{AB} considered the following $q$-series related to $S_1(q)$: 
\begin{equation}\label{A}
T_o(q):=\sum_{n\ge 0}t_o(n)q^n=\sum_{n\ge 0}\frac{(-q)_{2n}}{\left(q;q^2\right)_{n+1}}q^{2n}.
\end{equation}
Setting
\begin{equation}\label{B}
C(q):=(q)_{\infty}T_o(q)=:\sum_{n\ge 0}c(n)q^n,
\end{equation}
Andrews and El Bachraoui \cite[Theorem 3]{AB} showed that if $c(n)\neq 0$, then $24n+28$ is represented by the binary quadratic form $x^2+3y^2$. In the same spirit, they \cite[Conjecture 2]{AB} proposed the following conjecture:
\begin{conjecture}\label{conj2}
	If $8n+9$ is not represented by $x^2+2y^2$, then
	\[
	t_o(n)\equiv 0\pmod{2}.
	\]	
\end{conjecture}
Our main result establishes \Cref{conj2}.%% by proving the following theorem.
\begin{theorem}\label{thm}
 \Cref{conj2} holds. 
\end{theorem}

%{\bf \color{blue} K: added the following remark.}

\begin{remark}
After the completion and submission of the present manuscript, we became aware of the independent work \cite{LX}, which also proves \Cref{conj2}. The two works were carried out independently and simultaneously, and they employ substantially different methods.
\end{remark}

 The paper is organized as follows. In \Cref{sec1}, we record elementary $q$-series identities. Proof of \Cref{thm} is given in \Cref{sec2}. %We conclude the paper with a follow up problem stated in \Cref{sec3}. 

\section*{Acknowledgements}
The authors have received funding from the European Research Council (ERC) under the European Union's Horizon 2020 research and innovation programme (grant agreement No. 101001179).

\section{Preliminaries}\label{sec1}

Here we list three $q$-series identities which we use in \Cref{sec2}. We begin with the following representation of $C(q)$ \cite[equation (4.8)]{AB}:
\begin{proposition}\label{prop0}
We have
\[
2qC(q)=\Psi(q)T(q)-(q)_{\infty}\Theta(q),
\]
where %\footnote{The definitions are given in \cite[(4.3), below Proposition 1, (4.1)]{AB}.}
\begin{align}\label{def1}
\Psi(q)&:=\sum_{n\ge 0}q^{\frac{n(n+1)}{2}},\\\label{def2}
T(q)&:=\sum_{n\ge 0}\left(1-q^{2n+1}\right)q^{\frac{n(3n+1)}{2}},\\\label{def3}
\Theta(q)&:=\sum_{n\ge 0}(-1)^n q^{\frac{n(n+1)}{2}}.
\end{align}
\end{proposition}
Euler's pentagonal number theorem \cite[Corollary 1.7]{Andbook} states the following:
\begin{proposition}\label{prop1}
We have 
\[
(q)_{\infty}=\sum_{n\in \mathbb{Z}}(-1)^n q^{\frac{n(3n+1)}{2}}.
\]
\end{proposition}
Finally, we use the following eta-quotient representation of $\Psi(q)$ \cite[equation (2.2.13)]{Andbook}:
\begin{proposition}\label{prop2}
We have
\[
\Psi(q)=\frac{\left(q^2;q^2\right)^2_{\infty}}{(q)_{\infty}}.
\]
\end{proposition}

\section{Proof of \Cref{thm}}\label{sec2}

In this section, we prove \Cref{thm}.

\begin{proof}[Proof of \Cref{thm}]
By \eqref{A} and \eqref{B}, we have
\begin{equation}\label{def4}
T_o(q)=\frac{C(q)}{(q)_{\infty}}=\sum_{n\ge 0}t_o(n)q^n.
\end{equation}
Using \eqref{def4} and \Cref{prop0}, we obtain
\begin{align}\label{eqn1}
2\sum_{n\ge 0}t_o(n)q^{n+1}&=2qT_o(q)=\frac{2qC(q)}{(q)_{\infty}}%=\frac{\Psi(q)D(q)-(q)_{\infty}\Theta(q)}{(q)_{\infty}}\nonumber\\
%&
=\frac{\Psi(q)T(q)}{(q)_{\infty}}-\Theta(q).
\end{align}
Now, by \Cref{prop1}, we obtain
\begin{align*}
(q)_{\infty}&=\sum_{m\in \mathbb{Z}}(-1)^m q^{\frac{m(3m+1)}{2}}%=\sum_{m\ge 0}(-1)^m q^{\frac{m(3m+1)}{2}}+\sum_{m\ge 1}(-1)^m q^{\frac{m(3m-1)}{2}}\\
=\sum_{m\ge 0}(-1)^m q^{\frac{m(3m+1)}{2}}-\sum_{m\ge 0}(-1)^m q^{\frac{(m+1)(3m+2)}{2}}\\
&=\sum_{m\ge 0}(-1)^m\left(1-q^{2m+1}\right) q^{\frac{m(3m+1)}{2}}\\
&=\sum_{\substack{m\ge 0\\m\ \text{even}}}\left(1-q^{2m+1}\right) q^{\frac{m(3m+1)}{2}}-\sum_{\substack{m\ge 1\\m\ \text{odd}}}\left(1-q^{2m+1}\right) q^{\frac{m(3m+1)}{2}}\\
&=\sum_{m\ge 0}\left(1-q^{2m+1}\right) q^{\frac{m(3m+1)}{2}}-2\sum_{\substack{m\ge 1\\m\ \text{odd}}}\left(1-q^{2m+1}\right) q^{\frac{m(3m+1)}{2}}=T(q)-2S(q),
\end{align*}
by \eqref{def2}, where
\[
S(q):=\sum_{\substack{m\ge 1\\m\ \text{odd}}} \left(1-q^{2m+1}\right)q^{\frac{m(3m+1)}{2}}.
\]
Hence
\begin{equation}\label{eqn2}
T(q)=(q)_{\infty}+2S(q).
\end{equation}
Next, we have, by \eqref{def3} and \eqref{def1}
\begin{align}\label{eqn3}
\Theta(q)&=\sum_{n\ge 0}(-1)^n q^{\frac{n(n+1)}{2}}=\sum_{\substack{n\ge 0\\n\ \text{even}}} q^{\frac{n(n+1)}{2}}-\sum_{\substack{n\ge 0\\n\ \text{odd}}} q^{\frac{n(n+1)}{2}}\nonumber\\
&=\sum_{n\ge 0} q^{\frac{n(n+1)}{2}}-2\sum_{\substack{n\ge 1\\n\ \text{odd}}} q^{\frac{n(n+1)}{2}}=\Psi(q)-2R(q),
\end{align}
where
\[
R(q):=\sum_{\substack{n\ge 1\\n\ \text{odd}}} q^{\frac{n(n+1)}{2}}.
\]
Substituting \eqref{eqn2} and \eqref{eqn3} into \eqref{eqn1}, we obtain
\begin{align*}
2qT_o(q)&=\frac{\Psi(q)}{(q)_{\infty}}\left((q)_{\infty}+2S(q)\right)-\left(\Psi(q)-2R(q)\right)\nonumber\\
&=\frac{2\Psi(q)S(q)}{(q)_{\infty}}+2R(q).
\end{align*}
Multiplying both sides by $\frac 12$, this yields
\begin{equation}\label{eqn4}
qT_o(q)=\frac{\Psi(q)S(q)}{(q)_{\infty}}+R(q).
\end{equation}
By \Cref{prop2} and \Cref{prop1}
\begin{align*}
\frac{\Psi(q)}{(q)_{\infty}}&=\frac{\left(q^2;q^2\right)^2_{\infty}}{(q)^2_{\infty}}=(-q)^2_{\infty}\equiv \left(q^2;q^2\right)_{\infty}%\\%\pmod{2}\\%&
=\sum_{m\in \mathbb{Z}}(-1)^m q^{m(3m+1)}\\%\\%\pmod{2}%
&\equiv \sum_{m\in \mathbb{Z}} q^{m(3m+1)}\pmod{2}.
\end{align*}
Substituting this into \eqref{eqn4}, we have
\begin{align}\label{eqn5}
qT_o(q)&\equiv S(q)\sum_{m\in \mathbb{Z}} q^{m(3m+1)}+R(q)\nonumber\\%\pmod{2}\nonumber\\
&=\sum_{\substack{r\ge 0\\r\ \text{odd}}} \left(1-q^{2r+1}\right)q^{\frac{r(3r+1)}{2}}\sum_{m\in \mathbb{Z}} q^{m(3m+1)}+\sum_{\substack{m\ge 1\\m\ \text{odd}}} q^{\frac{m(m+1)}{2}}\nonumber\\
&\equiv\sum_{\substack{r\ge 1\\r\ \text{odd}\\m\in \mathbb{Z}}}\left(q^{\frac{r(3r+1)}{2}+m(3m+1)}+q^{\frac{r(3r+1)}{2}+(2r+1)+m(3m+1)}\right)+\sum_{\substack{m\ge 1\\m\ \text{odd}}} q^{\frac{m(m+1)}{2}}\nonumber\\
&=:\sum_{n\ge 1}\a(n)q^n\pmod{2}.
\end{align}
%where
%\[
%\sum_{n\ge 1}\a(n)q^n:=\sum_{\substack{r\ge 1\\r\ \text{odd}\\m\in \mathbb{Z}}}\left(q^{\frac{r(3r-1)}{2}+m(3m+1)}+q^{\frac{r(3r-1)}{2}+(2r+1)+m(3m+1)}\right)+\sum_{\substack{m\ge 1\\m\ \text{odd}}} q^{\frac{m(m+1)}{2}}.
%\]
Thus, by \eqref{def4} and \eqref{eqn5}, we have
\begin{align*}
qT_o(q)=\sum_{n\ge 0}t_o(n)q^{n+1}\equiv\sum_{n\ge 1}\a(n)q^n= \sum_{n\ge 0}\a(n+1)q^{n+1} \pmod{2},
\end{align*}
and so 
\[
t_o(n)\equiv \a(n+1)\pmod{2}.
\]

Now assume that 
\[
t_o(n)\not\equiv 0\pmod{2}.%\Leftrightarrow 
\]
Equivalently,
\[
\a(n+1)\not\equiv 0\pmod{2}.
\]
Then one of the terms on the right-hand side of \eqref{eqn5} contributes. We split into three cases.

 First, let $n+1=\frac{m(m+1)}{2}$ for some odd $m\in \mathbb{N}$. Then %Then, for some $m\in \mathbb{N}_0$, 
\[
 8n+9=4m^2+4m+1=(2m+1)^2.
\]
Thus, $8n+9$ is represented by $x^2+2y^2$ with $x=2m+1$ and $y=0$.  

Next, suppose that for some odd $r\in \mathbb{N}$ and some $m\in \mathbb{Z}$
\begin{align*}
n+1=\frac{r(3r+1)}{2}+m(3m+1).
\end{align*}
This implies that
\begin{align*}
 8n+9&=4r(3r+1)+8m(3m+1)+1=\left(2r+4m+1\right)^2+2\left(2r-2m\right)^2.
\end{align*}
Therefore, in this case, $8n+9$ is represented by $x^2+2y^2$ with $x=2r+4m+1$ and $y=2r-2m$.

Finally, assume that %for $r\in \mathbb{N}$, $r$ odd, and $m\in \mathbb{N}$,
\begin{align*}
&n+1=\frac{r(3r+1)}{2}+2r+1+m(3m+1)
\end{align*}
for some odd $r\in \mathbb{N}$ and some $m\in \mathbb{Z}$. Consequently, we have
\begin{align*}
8n+9&=4r(3r+1)+8(2r+1)+8m(3m+1)+1\\
&=\left(2r-4m+1\right)^2+2\left(2r+2m+2\right)^2.
\end{align*}
Therefore, $8n+9$ is represented by $x^2+2y^2$ with $x=2r-4m+1$ and $y=2r+2m+2$. It follows from \eqref{eqn5} that whenever $8n+9$ is not represented by $x^2+2y^2$,
\[
t_o(n)\equiv 0\pmod{2}.
\]
This proves \Cref{thm}.\qedhere 
\end{proof}


\begin{thebibliography}{999}
	
\bibitem{Andbook} G. Andrews, {\it The theory of partitions}, Cambridge University Press, Cambridge, 1998.

\bibitem{AndCC} G. Andrews, {\it Concave and convex compositions}, Ramanujan J. {\bf 31} (2013), 67--82.

\bibitem{AB} G. Andrews and M. El Bachraoui, {\it On a two-color partition series and its companions}, arXiv:2606.30208 (2026).

\bibitem{LX} E. Liu and E. Xia, {\it A proof of Andrews--El Bachraoui's conjecture on the parity of coefficients of a $q$-series}, arXiv:2607.07145 (2026).




%\bibitem{Ono} K. Ono, {\it The web of modularity: arithmetic of the coefficients of modular forms and $q$-series}, CBMS Regional Conference Series in Mathematics {\bf 102} (2003), AMS, Providence, RI.





  
\end{thebibliography}
\end{document}